\documentclass[11pt]{amsart}

\usepackage{amsmath}
  \usepackage{paralist}
  \usepackage{graphics} 
  \usepackage{epsfig} 
\usepackage{graphicx}  \usepackage{epstopdf}

 \usepackage[colorlinks=true]{hyperref}
\hypersetup{urlcolor=blue, citecolor=red}

\newtheorem{theorem}{Theorem}[section]

\newtheorem*{main*}{Main Theorem}
\newtheorem{lemma}[theorem]{Lemma}
\newtheorem{proposition}[theorem]{Proposition}

\newtheorem{question}[theorem]{Question}
\theoremstyle{definition}
\newtheorem{definition}[theorem]{Definition}
\newtheorem{remark}[theorem]{Remark}

\def\BB{{\mathcal{B}}}
\def\a{\alpha}
\def\b{\beta}
   
\def\d{\delta}   
   
 \def\e{\epsilon}

\def\vp{\varphi}

\def\ae{\text{-a.e.}\ }

\newcommand{\RR}{\mathbb{R}}

\newcommand{\NN}{\mathbb{N}}

\newcommand{\ZZ}{\mathbb{Z}}

\newcommand{\cC}{\mathcal{C}}

\newcommand{\cP}{\mathcal{P}}
\newcommand{\cB}{\mathcal{B}}
\newcommand{\cV}{\mathcal{V}}

\newcommand{\cE}{\mathcal{E}}
\newcommand{\cM}{\mathcal{M}}
\newcommand{\M}{\mathcal{M}}
\newcommand{\cU}{\mathcal{U}}
\newcommand{\cL}{\mathcal{L}}

\newcommand{\diam}{\mathrm{diam}}
\newcommand{\mdim}{\mathrm{mdim}}

\title[Running heading with forty characters or less]
      {On relative metric mean dimension with potential and variational principles}

\author[first-name1 last-name1 and first-name2 last-name2]{Weisheng Wu}



\thanks{}

\address{W. Wu: Department of Applied Mathematics, College of Science, China Agricultural University, Beijing, 100083, P.R. China}
 \email{wuweisheng@cau.edu.cn}

\begin{document}

\maketitle
\markboth{}
{}
\renewcommand{\sectionmark}[1]{}
\begin{abstract}
In this article, we introduce a notion of relative mean metric dimension with potential for a factor map $\pi: (X,d, T)\to (Y, S)$ between two topological dynamical systems. To link it with ergodic theory, we establish four variational principles in terms of metric entropy of partitions, Shapira's entropy, Katok's entropy and Brin-Katok local entropy respectively. Some results on local entropy with respect to a fixed open cover are obtained in the relative case. We also answer an open question raised by Shi \cite{Shi} partially for a very well-partitionable compact metric space, and in general we obtain a variational inequality involving box dimension of the space.
Corresponding inner variational principles given an invariant measure of $(Y,S)$ are also investigated.
\end{abstract}

\section{Introduction}
Let $(X, T)$ be a topological dynamical system (TDS for short), that is, $X$ is a compact metrizable space and $T: X\to X$ is a homeomorphism. The topological entropy $h_{\text{top}}(T)$ and metric entropy $h_\mu(T)$ measure the complexity of the dynamical system from topological and measure-theoretic points of view respectively. The well known variational principle established by Goodwyn \cite{Go} and Goodman \cite{Good} states that
$$h_{\text{top}}(T)=\sup_{\mu\in \M_T(X)}h_\mu(T)$$
 where $\M(X)$ (resp. $\M_T(X)$, $\cE_T(X)$) denotes the set of all (resp. $T$-invariant, $T$-ergodic) probability Borel measures on $X$.

To further quantify the complexity of dynamical systems of infinite entropy, a new invariant called topological mean dimension was introduced by Gromov \cite{Gro}. Metric mean dimension was introduced later by Lindenstrauss and Weiss \cite{LW00} as an upper bound for topological mean dimension. Recently, a double (minimax) variational principle is obtained for topological mean dimension by Lindenstrauss and Tsukamoto \cite{LT}, and is extended to topological mean dimension with potential by Tsukamoto \cite{Ts}. Lindenstrauss and
Tsukamoto \cite{LT1} also established a variational principle formulating metric mean dimension as a supremum of certain rate distortion
functions over invariant measures of the system. Alternative formulations of metric mean dimension in terms of metric entropy of partitions, Brin-Katok local entropy, Katok's entropy and Shapira's entropy are developed by Velozo-Velozo \cite{VV}, Gutman and \'{S}piewak \cite{GS}, and Shi \cite{Shi}.

In this paper, we consider a factor map $\pi: (X,d, T)\to (Y, S)$, i.e., $T: X\to X$ and $S: Y\to Y$ are two topological dynamical systems and $\pi: X \to Y$ is a continuous, surjective map such that $\pi \circ T=S \circ\pi$. Let $C(X, \RR)$ denote the set of all continuous functions $\vp: X\to \RR$.
The following classical relative variational principle for pressure was obtained by Ledrappier and Walters \cite{LW} and by Downarowicz and Serafin \cite{DS}: For any $\nu\in \mathcal{M}_S(Y)$ and any potential $\vp \in C(X, \RR)$,
\begin{equation}\label{e:LW}
\begin{aligned}
\sup_{\mu\in \mathcal{M}_T(X)} \Big\{h_\mu(T, X|Y)+\int_X \vp(x) d\mu(x): \pi \mu=\nu\Big\}=\int_Y P(T, \vp, y)d\nu(y)
\end{aligned}
\end{equation}
and
\begin{equation*}\label{e:LW2}
\begin{aligned}
\sup_{\mu\in \mathcal{M}_T(X)} \Big\{h_\mu(T, X|Y)+\int_X \vp(x) d\mu(x)\Big\}=P(T, \vp|Y).
\end{aligned}
\end{equation*}
This motivates us to introduce a notion of relative metric mean dimension with potential $\vp\in C(X,\RR)$ for a factor map $\pi: (X,d, T)\to (Y, S)$ (See Definition \ref{MMD}). When $(Y,S)$ is a trivial system and $\vp\equiv 0$, this notion recovers the metric mean dimension of $(X, d, T)$. We establish four variational principles in terms of metric entropy of partitions (Theorem \ref{VPpartition}), Shapira's entropy (Theorem \ref{VPShapira}), Katok's entropy (Theorem \ref{VPKatok}) and Brin-Katok local entropy (Theorem \ref{VPBK}) respectively. We also investigate relative metric mean dimension with potential with respect to an invariant measure of $(Y,S)$ of Ledrappier-Walters type \eqref{e:LW} and establish corresponding inner variational principles. We highlight that in our variational principles ``$\sup$'' actually becomes ``$\max$''.

A key ingredient in the proof of the variational principles, especially in terms of metric entropy of partitions, is the local versions of the variational principle with respect to a fixed open cover. In recent years, local variational principles have been extensively studied, for example in \cite{BGH, Rom, GW, HY, HYi}, etc. Relative versions of local variational principles for entropy and also for pressure are obtained in \cite{HYZ, MaChen}, etc. We use the relative local variational principle for pressure obtained in \cite{MaChen}.

To establish variational principles in terms of Shapira's entropy, Katok's entropy and Brin-Katok local entropy, we investigate the relations between these types of entropy at a fixed resolution $\e>0$. We extend Shapira's theorem to the relative case (Proposition \ref{shapirathm}). An open question raised by Shi \cite{Shi} is also considered: We obtain a variational inequality (Theorem \ref{VPBK1}) involving the lower Brin-Katok local entropy when the underlying metric space has finite box dimension, and a variational principle  (Theorem \ref{VPBK2}) when $(X,d)$ is very well-partitionable. We explore the ideas from \cite{BK} and its relative version in \cite{Zhou} to compare the measures of Bowen balls and the elements from refining partitions. However, the problem in general is still open.

The paper is organized as follows. In Section $2$, we give a precise definition for relative metric mean dimension with potential. In Section $3$, variational principles in terms of metric entropy of partitions, Shapira's entropy and Katok's entropy respectively are proved. In Section $4$, the variational principles or inequality involving upper or lower Brin-Katok local entropy are obtained. In Section $5$, we discuss relative metric mean dimension with potential with respect to an invariant measure of $(Y,S)$ of Ledrappier-Walters' type \eqref{e:LW} and corresponding inner variational principles. A proof of Proposition \ref{shapirathm} is included in the appendix.

\section{Relative metric mean dimension with potential}

Let $(X,d)$ be a compact metric space and $T:X\to X$ a homeomorphism. There are three equivalent ways using separated sets, spanning sets and open covers respectively, to define topological pressure and metric mean dimension with potential.
We first prepare some necessary notations and at the end of this section we give a precise definition for relative metric mean dimension with potential.

\subsection{Topological pressure}
For $n\in \NN$, define the \emph{Bowen metric} as
$$d_{n}(x,y):=\max _{0 \leq j \leq n-1}d(T^j(x),T^j(y)).$$
Let $K\subset X$ and $\e>0$. A subset $E\subset K$ is called \emph{$(n,\epsilon)$ separated} if $d_n(x,y)>\e$ for any $x,y\in E, x\neq y$. For any potential $\vp\in C(X, \RR)$, define
\begin{equation*}
\begin{aligned}
P_n(d, T, \varphi, \epsilon, K):=\sup\Big\{&\sum_{y\in E}\exp((S_n\varphi)(y)):\\
&E \text{\ is an\ } (n, \epsilon) \text{\ separated subset of\ }K\Big\}
\end{aligned}
\end{equation*}
where $(S_n\varphi)(y)=\sum_{i=0}^{n-1}\varphi^i(y)$. Then we put
\begin{equation*}
\begin{aligned}
P(d, T, \varphi, \epsilon, K):=\limsup_{n\to \infty}\frac{1}{n}\log P_n(d, T, \varphi, \epsilon, K).
\end{aligned}
\end{equation*}
If $K$ is dropped from the notation, it always means $K=X$. We note that for any $\vp, \phi\in C(X,\RR)$ (see \cite[Theorem 9.7(iv)]{Wa})
\begin{equation}\label{e:difference}
\begin{aligned}
|P(d, T, \varphi, \epsilon, K)-P(d, T, \phi, \epsilon, K)|\le \|\vp-\phi\|
\end{aligned}
\end{equation}
where $\| \cdot\|$ denote the $C^0$ norm of $C(X,\RR)$.
\begin{definition}\label{pressure}
We define the \emph{topological pressure} of $T$ with respect to potential $\varphi$ on $K\subset X$ to be
\begin{equation*}
\begin{aligned}
P(T, \varphi, K)&:=\lim_{\e \to 0}P(d, T, \varphi, \epsilon, K).
\end{aligned}
\end{equation*}
\end{definition}
We note that $P(T, \varphi, K)$ is independent of the choice of metric $d$. When $K=X$, $P(T, \varphi)$ is called the \emph{topological pressure} of $T$ with respect to $\varphi$. When $\vp\equiv 0$, $P(T, 0)$ reduces to the \emph{topological entropy} of $T$.

A set $F \subset K$ is called an \emph{$(n,\epsilon)$ spanning set }of $K$ if $K \subset \bigcup_{y\in F}B_{n}(y,\epsilon)$, where $B_{n}(y,\epsilon):=\{z\in X: d_{n}(y,z) <\epsilon\}$ is the $(n,\epsilon)$ \emph{Bowen ball} around $y$. Put
$$Q(d, T, \varphi,\epsilon, K):=\limsup_{n\to \infty}\frac{1}{n}\log Q_n(T,\varphi, \epsilon, K)$$
where
\begin{equation*}
\begin{aligned}
Q_n(d, T,\varphi, \epsilon, K):=\inf\Big\{&\sum_{y\in F}\exp((S_n\varphi)(y)):\\
&F \text{\ is an\ } (n, \epsilon) \text{\ spanning subset of\ }K\Big\}.
\end{aligned}
\end{equation*}

Let $\cB(X)$ be the collection of all Borel subsets of $X$. A \emph{cover} of $X$ is a finite family of Borel subsets of
$X$, whose union is $X$. A \emph{partition} of $X$ is a cover of $X$ whose elements are pairwise disjoint.
Let $\cP_X$ denote the set of partitions of $X$, $\mathcal{C}_X$ the set of covers of $X$ and $\mathcal{C}_X^o\subset \mathcal{C}_X$ the set of open covers of $X$. Given two covers $\cU, \cV \in \cC_X$, $\cU$ is said to be finer than $\cV$, denoted by $\cU \succeq \cV$ or $\cV \preceq \cU$,  if each element of
$\cU$ is contained in some element of $\cV$. Let $\cU\vee \cV:=\{U\cap V : U \in \cU, V\in \cV\}$.
For $m,n\in \ZZ\cup\{\pm \infty\}$, denote $\mathcal{U}_m^n:=\bigvee_{i=m}^n T^{-i}\mathcal{U}$. For $\cU\in \cC_X$, let $\diam \cU$ denote the maximal diameter of elements in $\cU$. The Lebesgue number of $\cU\in \cC^o_X$ is denoted by $Leb \cU$.

For $\cU\in \cC^o_X$ put
\begin{equation*}
\begin{aligned}
p_n(d, T,\varphi, \cU, K):=\inf\Big\{&\sum_{B\in \mathcal{V}} \sup_{y\in B}\exp((S_n\varphi)(y)): \\ &\mathcal{V} \text{\ is a finite subcover of \ } \mathcal{U}_0^{n-1} \text{\ and} \bigcup_{V\in \mathcal{V}}\supset K\Big\}
\end{aligned}
\end{equation*}
and
$$p(d, T,\varphi, \cU, K):=\limsup_{n\to \infty}\frac{1}{n}\log p_n(d, T,\varphi, \cU, K).$$
The above ``$\limsup$'' becomes ``$\lim$'' if $K$ is $T$-invariant by Theorem 11.5 in \cite{Pesin}.

The following lemma is straightforward but useful. For its proof, we refer to Theorem 9.2, p.209 and p.211 in \cite{Wa}.
\begin{lemma}\label{coincide}
Let $K\subset X$ and $\vp\in C(X,\RR)$. We have:
\begin{enumerate}
  \item If $\e>0$ and $\cU$ is an open cover with $\diam \cU\le \e$, then
  $$Q_n(d, T,\varphi, \epsilon, K)\le P_n(d, T,\varphi, \epsilon, K)\le p_n(d, T,\varphi, \cU, K).$$
  \item If $\cU$ is an open cover with $Leb \cU=\d$, then
  $$p_n(d, T,\varphi, \cU, K)\le e^{n\tau_\cU}Q_n(d, T,\varphi, \d/2, K)\le e^{n\tau_\cU}P_n(d, T,\varphi, \d/2, K)$$
  where $\tau_\cU:=\sup\{|\vp(x)-\vp(y)|: d(x,y)\le \diam \cU\}$.
  \item If $\d=\d(\e)>0$ is such that $d(x,y)<\e/2$ implies $|\vp(x)-\vp(y)|<\d$, then
  $$P_n(d, T,\varphi, \epsilon, K)\le e^{n\d(\e)}Q_n(d, T,\varphi, \epsilon/2, K).$$
  \end{enumerate}
\end{lemma}
As an immediate corollary of Lemma \ref{coincide}, one has the following reformulations of topological pressure.
\begin{proposition}
Let $K\subset X$ and $\vp\in C(X,\RR)$. We have
\begin{equation*}
\begin{aligned}
P(T, \varphi, K)=\lim_{\e \to 0}Q(d, T, \varphi, \epsilon, K)
=\lim_{\diam \cU\to 0}p(d, T,\varphi, \cU, K).
\end{aligned}
\end{equation*}
\end{proposition}

\subsection{Relative metric mean dimension with potential}

Let $T:(X,d)\to (X,d)$ and $S:Y\to Y$ be two TDSs. Recall that a continuous map $\pi:X\to Y$ is called a factor map between $(X,T)$ and $(Y,S)$ if it is onto and $\pi\circ T=S\circ\pi$. Given $\e>0$ and $\mathcal{U}\in \mathcal{C}^o(X)$, write
\begin{equation*}
\begin{aligned}
P_n(d, T,\varphi,\e|Y):=& \sup_{y\in Y}P_n(d, T,\varphi,\e, \pi^{-1}y),\\
Q_n(d, T,\varphi,\e|Y):=&  \sup_{y\in Y}Q_n(d, T,\varphi,\e, \pi^{-1}y),\\
p_n(d, T,\varphi,\cU|Y):=&  \sup_{y\in Y}p_n(d, T,\varphi,\cU, \pi^{-1}y).
\end{aligned}
\end{equation*}
Then we define
\begin{equation*}
\begin{aligned}
P(d, T,\varphi,\e|Y):=& \limsup_{n\to \infty}\frac{1}{n}\log P_n(d, T,\varphi,\e|Y) ,\\
Q(d, T,\varphi,\e|Y):=&  \limsup_{n\to \infty}\frac{1}{n}\log Q_n(d, T,\varphi,\e|Y),\\
P(d, T,\varphi,\cU|Y):=&  \lim_{n\to \infty}\frac{1}{n}\log p_n(d, T,\varphi,\cU|Y).
\end{aligned}
\end{equation*}
The limit in the last equality exists by the subadditivity of $\log p_n(d, T,\varphi,\cU|Y)$, see \cite{MaChen}. $P(d, T,\varphi,\cU|Y)$
is called the \emph{relative local topological pressure} of $\vp$ with respect to $\cU$ and $(Y,S)$. The \emph{relative topological pressure} of $\vp$ with respect to $(Y,S)$ is then defined as
$$P(T,\varphi|Y)=\sup_{\cU\in \cC^o_X}P(d, T,\varphi,\cU|Y).$$
By Lemma \ref{coincide},
$$P(T,\varphi|Y)=\lim_{\e\to 0}P(d, T,\varphi,\e|Y)=\lim_{\e\to 0}Q(d, T,\varphi,\e|Y)$$
and it is independent of the choice of metric $d$.
For given $y\in Y$, we also define
\begin{equation*}
\begin{aligned}
P(d, T,\varphi,\e, y):=& \limsup_{n\to \infty}\frac{1}{n}\log P_n(d, T,\varphi,\e, \pi^{-1}y) ,\\
Q(d, T,\varphi,\e, y):=&  \limsup_{n\to \infty}\frac{1}{n}\log Q_n(d, T,\varphi,\e, \pi^{-1}y),\\
P(d, T,\varphi,\cU, y):=&  \limsup_{n\to \infty}\frac{1}{n}\log p_n(d, T,\varphi,\cU, \pi^{-1}y).
\end{aligned}
\end{equation*}

Now we are ready to define relative metric mean dimension with potential.
\begin{definition}\label{MMD}
Let $\pi:(X,d,T)\to (Y,S)$ be a factor map and $\vp\in C(X,\RR).$ The \emph{upper relative metric mean dimension with potential} $\vp$ of $(X,d,T)$ with respect to $(Y, S)$ is defined by
\begin{equation*}
\begin{aligned}
\overline{\mdim}_M(X,d,T,\vp|Y):=\limsup_{\e \to 0}\frac{1}{\log\frac{1}{\e}}P(d, T, \varphi\log\frac{1}{\e}, \epsilon|Y).
\end{aligned}
\end{equation*}
Similarly, the \emph{lower relative metric mean dimension with potential} $\vp$ of $(X,d,T)$ with respect to $(Y, S)$ is defined by
\begin{equation*}
\begin{aligned}
\underline{\mdim}_M(X,d,T,\vp|Y):=\liminf_{\e \to 0}\frac{1}{\log\frac{1}{\e}}P(d, T, \varphi\log\frac{1}{\e}, \epsilon|Y).
\end{aligned}
\end{equation*}
\end{definition}
When $(Y,S)$ is trivial, we obtain \emph{upper/lower metric mean dimension with potential}, which is introduced and studied in \cite{Ts}.

\begin{remark}\label{compare}
By definition, $\overline{\mdim}_M(X,d,T,\vp|Y)$ depends on the metric $d$ on $X$, but not on the metric on $Y$. It is clear that
$$\overline{\mdim}_M(X,d,T,\vp|Y)\le \overline{\mdim}_M(X,d,T,\vp).$$
Unlike the topological entropy case, it is possible that there exists some metric $d'$ on $Y$ such that $\overline{\mdim}_M(X,d,T) < \overline{\mdim}_M(Y,d',S)$.
\end{remark}

\begin{proposition}
We have
\begin{equation*}
\begin{aligned}
\overline{\mdim}_M(X,d,T,\vp|Y)=&\limsup_{\e \to 0}\frac{1}{\log\frac{1}{\e}}Q(d, T, \varphi\log\frac{1}{\e}, \epsilon|Y),\\
\underline{\mdim}_M(X,d,T,\vp|Y)=&\liminf_{\e \to 0}\frac{1}{\log\frac{1}{\e}}Q(d, T, \varphi\log\frac{1}{\e}, \epsilon|Y).
\end{aligned}
\end{equation*}
\end{proposition}
\begin{proof}
By (1) of Lemma \ref{coincide}, we have
\begin{equation*}
\begin{aligned}
\overline{\mdim}_M(X,d,T,\vp|Y)\ge &\limsup_{\e \to 0}\frac{1}{\log\frac{1}{\e}}Q(d, T, \varphi\log\frac{1}{\e}, \epsilon|Y).
\end{aligned}
\end{equation*}
By (3) of Lemma \ref{coincide} and \eqref{e:difference}, we have
\begin{equation*}
\begin{aligned}
\overline{\mdim}_M(X,d,T,\vp|Y)\le &\limsup_{\e \to 0}\frac{1}{\log\frac{1}{\e}}\left(P(d, T, \varphi\log\frac{2}{\e}, \epsilon|Y)+\|\vp\|\log2\right)\\
\le &\limsup_{\e \to 0}\frac{1}{\log\frac{1}{\e}}\left(\log e^{\d(\e)\log \frac{2}{\e}}+Q(d, T, \varphi\log\frac{2}{\e}, \epsilon/2|Y)\right)\\
= &\limsup_{\e \to 0}\frac{1}{\log\frac{1}{\e}}Q(d, T, \varphi\log\frac{1}{\e}, \epsilon|Y).
\end{aligned}
\end{equation*}
where we used the facts that $\lim_{\e\to 0}\d(\e)=0$ and $\lim_{\e\to 0}\frac{\log\frac{1}{\e}}{\log\frac{2}{\e}}=1$.
The equality for $\underline{\mdim}_M(X,d,T,\vp|Y)$ in the proposition follows analogously.
\end{proof}

\section{Three variational principles}
Our goal in this section is to prove three variational principles, Theorems \ref{VPpartition}, \ref{VPShapira}, \ref{VPKatok}. To begin with, we must reacall some standard notations and classical results on measurable partitions and conditional metric entropy.

\subsection{Conditional metric entropy and local variational principle}
Let $(X, \mathcal{A}, \nu)$ be a standard probability space. For a partition (not necessarily countable) $\xi$ of $X$, let $\xi(x)$ denote the element of $\xi$ containing $x$.
Let $\mathcal{B}(\xi)$ denote the smallest
sub-$\sigma$-algebra of $\mathcal{A}$ that contains all elements of $\xi$.
A partition $\xi$ of $X$ is called \emph{measurable} if there exists a countable set $\{A_n\}_{n\in \mathbb{N}}\subset \mathcal{B}(\xi)$ such that
for almost every pair $C_1,C_2\in \xi$, we can find some $A_n$ which separates them
in the sense that $C_1\subset A_n, C_2\subset X-A_n$.
The \emph{canonical system
of conditional measures of $\nu$ relative  to $\xi$}
is a family of probability
measures $\{\nu_x^\xi: x\in X\}$ with $\nu_x^\xi\bigl(\xi(x)\bigr)=1$,
such that for every measurable set $B\subset X$, $x\mapsto \nu_x^{\xi}(B)$
is $\mathcal{B}(\xi)$-measurable and
\[
\nu (B)=\int_X\nu_x^{\xi}(B)d\nu(x).
\]
The classical result of Rokhlin (cf. \cite{R}) says that if $\xi$ is a measurable partition, then there exists a system of conditional measures relative to $\xi$. It is unique
in the sense that two such systems coincide in a set of full $\nu$-measure.
For measurable partitions $\a$ and $\xi$, the \emph{conditional information function} is defined as
$$I_\nu(\a|\xi)(x):=-\log \nu_x^\xi(\alpha(x))$$
and the \emph{conditional entropy of $\a$ given $\xi$ }with respect to $\nu$ is defined as
$$H_\nu(\a|\xi):=\int_X I_\nu(\a|\xi)(x)d\nu(x).$$

Let $T: (X, \mathcal{A}, \mu)\to (X, \mathcal{A}, \mu)$ be a measure-preserving transformation.
The \emph{conditional entropy of $T$ with respect to a measurable partition $\a$
given $\xi$} is defined as
$$h_\mu(T, \alpha|\xi):=\limsup_{n\to \infty}\frac{1}{n}H_\mu(\alpha_0^{n-1}|\xi).
$$
A measurable partition $\xi$ is called \emph{invariant} under $T$ if $T^{-1}(\xi(Tx))=\xi(x)$ for $\mu \ae x.$ The following is a conditional version of Shannon-McMillan-Breiman (SMB for short) theorem.
\begin{theorem}\cite[Theorem B.0.1]{Do}\label{smb0}
Let $(X,T)$ be a TDS, $\mu\in\M_T(X)$, $\xi$ a $T$-invariant measurable partition and $\alpha\in \cP_X$.
Then the following limit exists
$$
\lim_{n\rightarrow\infty}\frac{1}{n}I_{\mu}(\alpha_0^{n-1}|\xi)(x):=h_\mu(T,\a|\xi,x) \quad   \mu \ae \text{and in } L^1(\mu),
$$
and $\int_X h_\mu(T,\a|\xi,x) d\mu(x)=h_\mu(T,\a|\xi).$ Moreover, if $\mu$ is ergodic then
$$h_\mu(T,\a|\xi,x)=h_\mu(T,\a|\xi)$$
 for $\mu \ae x\in X$.
\end{theorem}

Let $T:X\to X$ and $S:Y\to Y$ be two TDSs and $\pi:X\to Y$ a factor map.
We proceed by defining the conditional metric entropy of $\mu\in \M_T(X)$ with respect to $(Y,S)$. Let $\epsilon_Y$ denote the partition of $Y$ into points and $\BB(Y)$ the Borel $\sigma$-algebra of $Y$. Then in the Lebesgue space $(X,\mu)$, the measurable partition $\pi^{-1}\epsilon_Y$ generates the $\sigma$-algebra $\pi^{-1}\BB(Y)$ and it is $T$-invariant. We simply denote $\mu_x:=\mu_x^{\pi^{-1}\epsilon_Y}$ and also $\mu_y:=\mu_x$ if $\pi x=y$, for $\mu \ae x$. Given $\a\in \cP_X$, define the conditional entropy of $\a$ with respect to $(Y,S)$ as
$$h_\mu(T,\a|Y):=h_\mu(T,\a|\pi^{-1}\epsilon_Y)=\lim_{n\to \infty}\frac{1}{n}H_\mu(\a_0^{n-1}|\pi^{-1}\epsilon_Y).$$
The above limit exists since the sequence $b_n=H_\mu(\a_0^{n-1}|\pi^{-1}\epsilon_Y)$ is subadditive. Moreover, the \emph{conditional metric entropy of $(X,T,\mu)$ with respect to $(Y,S)$} is defined by
$$h_\mu(T,X|Y):=\sup_{\a\in \cP_X} h_\mu(T,\a|Y).$$
Applying SMB Theorem \ref{smb0} for $\xi=\pi^{-1}\e_Y$, we can write
$$
h_\mu(T,\a|Y,x):=\lim_{n\rightarrow\infty}\frac{1}{n}I_{\mu}(\alpha_0^{n-1}|\pi^{-1}\e_Y)(x),
$$
which exists for $\mu \ae x\in X$ and satisfies that $\int_X h_\mu(T,\a|Y,x) d\mu(x)=h_\mu(T,\a|Y).$

We have the following classical conditional variational principle by Ledrappier and Walters.
\begin{theorem}[\cite{LW,DS}]\label{conditionalVP}
Let $T: (X, d)\to (X,d)$ and $S:Y\to Y$ be two TDSs and $\pi:X\to Y$ a factor map. Then
 for any $\nu\in \mathcal{M}_S(Y)$ and any potential $\vp \in C(X, \RR)$,
\begin{equation*}
\begin{aligned}
\sup_{\mu\in \mathcal{M}_T(X)} \Big\{h_\mu(T, X|Y)+\int_X \vp(x) d\mu(x): \pi \mu=\nu\Big\}=\int_Y P(T, \vp, y)d\nu(y).
\end{aligned}
\end{equation*}
and
\begin{equation*}
\begin{aligned}
\sup_{\mu\in \mathcal{M}_T(X)} \Big\{h_\mu(T, X|Y)+\int_X \vp(x) d\mu(x)\Big\}=P(T, \vp|Y).
\end{aligned}
\end{equation*}
\end{theorem}

The \emph{local conditional metric entropy with respect to $\cU\in \cC_X^o$} is introduced in \cite{HYZ} as
$$h_\mu(T,\cU|Y):=\inf_{\a\in \cP_X, \a\succeq \cU}h_\mu(T,\a|Y).$$
Let $\mu=\int_{\mathcal{\cE}_T(X)}\nu d\tau(\nu)$ be the unique ergodic decomposition
where $\tau$ is a probability measure on the Borel subsets of $\mathcal{M}_T(X)$ and $\tau(\mathcal{E}_T(X))=1$.
Then by Theorem 5.3 in \cite{HYZ},
\begin{equation*}\label{e:ergodicdecom}
h_\mu(T,\cU|Y)+\int_X \varphi d\mu=\int_{\mathcal{E}_T(X)}\Big(h_\nu(T,\cU|Y)+\int_X \varphi d\nu\Big) d\tau(\nu).
\end{equation*}
Combining it with Theorem 1.3 in \cite{MaChen}, we have the following local variational principle  which plays a key role in our proof below.
\begin{theorem} \cite[Theorem 1.3]{MaChen}\label{LVP}
Let $\pi:(X, T)\to (Y,S)$ be a factor map between two TDSs and $\cU\in \cC_X^o$. Then
\begin{equation*}
\begin{aligned}
P(T,\vp,\cU|Y)=\max_{\mu\in \M_T(X)}\left\{h_\mu(T,\cU|Y)+\int \vp d\mu\right\}.
\end{aligned}
\end{equation*}
Moreover, $\M_T(X)$ in the above equality can be replaced by $\cE_T(X)$.
\end{theorem}

\subsection{Variational principle via entropy of partitions}
Our first variational principle for relative metric mean dimension with potential is as follows.
\begin{theorem}\label{VPpartition}
Let $\pi:(X, d, T)\to (Y,S)$ be a factor map between two TDSs. Then
for any potential $\vp \in C(X, \RR)$,
\begin{equation*}
\begin{aligned}
&\overline{\mdim}_M(X,d,T,\vp|Y)\\
=&\limsup_{\e \to 0}\frac{1}{\log\frac{1}{\e}}\max_{\mu\in \M_T(X)}\inf_{\diam \a\le \e}\left(h_\mu(T,\a|Y)+\log\frac{1}{\e}\int \vp d\mu\right)
\end{aligned}
\end{equation*}
and
\begin{equation*}
\begin{aligned}
&\underline{\mdim}_M(X,d,T,\vp|Y)\\
=&\liminf_{\e \to 0}\frac{1}{\log\frac{1}{\e}}\max_{\mu\in \M_T(X)}\inf_{\diam \a\le \e}\left(h_\mu(T,\a|Y)+\log\frac{1}{\e}\int \vp d\mu\right).
\end{aligned}
\end{equation*}
Moreover, both $\M_T(X)$ in the above equalities can be replaced by $\cE_T(X)$.
\end{theorem}

\begin{proof}
Given $\e>0$, choose $\cU=\cU_\e\in \cC^o_X$ such that $\diam \cU \le \e$ and $Leb \cU\ge \e/4$ (see Lemma 3.4 in \cite{GS}). By Theorem \ref{LVP}, (2) of Lemma \ref{coincide} and \eqref{e:difference}
\begin{equation*}
\begin{aligned}
&\sup_{\mu\in \M_T(X)}\inf_{\diam \a\le \e}(h_\mu(T,\a|Y)+\log\frac{1}{\e}\int \vp d\mu)\\
\le &\max_{\mu\in \M_T(X)}\inf_{\a\succeq \cU}(h_\mu(T,\a|Y)+\log\frac{1}{\e}\int \vp d\mu)\\
=&P(T,\vp\log\frac{1}{\e},\cU|Y)\\
=&\lim_{n\to \infty}\frac{1}{n}\log \sup_{y\in Y}p_n(d, T, \varphi\log\frac{1}{\e}, \cU, \pi^{-1}y)\\
\le & \lim_{n\to \infty}\frac{1}{n}\log \sup_{y\in Y}(1/\e)^{n\tau_\cU}P_n(d, T, \varphi\log\frac{1}{\e}, Leb\cU/2, \pi^{-1}y)\\
\le & \lim_{n\to \infty}\frac{1}{n}\log \sup_{y\in Y}(1/\e)^{n\tau_\cU}P_n(d, T, \varphi\log\frac{1}{\e}, \e/8, \pi^{-1}y)\\
=&P(d, T, \varphi\log\frac{8}{\e}, \epsilon/8|Y)+\|\vp\|\log 8+\tau_\cU\log\frac{1}{\e}.
\end{aligned}
\end{equation*}
On the other hand, by (1) of Lemma \ref{coincide} and the ``moreover'' part of Theorem \ref{LVP} we have
\begin{equation*}
\begin{aligned}
&P(d, T, \varphi\log\frac{1}{\e}, \epsilon|Y)\\
= & \lim_{n\to \infty}\frac{1}{n}\log \sup_{y\in Y}P_n(d, T, \varphi\log\frac{1}{\e}, \e, \pi^{-1}y)\\
\le&\lim_{n\to \infty}\frac{1}{n}\log \sup_{y\in Y}p_n(d, T, \varphi\log\frac{1}{\e}, \cU, \pi^{-1}y)\\
=&P(T,\log\frac{1}{\e}\vp,\cU|Y)\\
= &\max_{\mu\in \cE_T(X)}\inf_{\a\succeq \cU}\left(h_\mu(T,\a|Y)+\log\frac{1}{\e}\int \vp d\mu\right)\\
\le &\sup_{\mu\in \cE_T(X)}\inf_{\diam \a\le \e/8}\left(h_\mu(T,\a|Y)+\log\frac{1}{\e}\int \vp d\mu\right)
\end{aligned}
\end{equation*}
where in the last equality we used the fact that if $\diam \a\le \e/8\le Leb \cU$, then $\a\succeq \cU$.
Note that as $\e\to 0$, $\tau_\cU\to 0$, and $\lim_{\e\to 0}\frac{\log\frac{1}{\e}}{\log\frac{C}{\e}}$ for any $C>0$. Dividing both sides by $\log\frac{1}{\e}$ and taking the limit as $\e\to 0$, we obtain the variational principle with ``$\sup$'', which becomes a ``$\max$'' by the very last inequality above.
\end{proof}

\subsection{Variational principle via Shapira's entropy}
For $\nu\in \cM(X)$, $\cU\in \cC_X$ and $0<\rho<1$, let $N_\nu(\cU,\rho)$ denote the minimal number of elements of $\cU$, needed to cover a subset of $X$ whose $\nu$-measure is at least $1-\rho$.
Let $\mu\in \cE_T(X)$ and $\cU\in \cC^o_X$. Define
\begin{equation*}
\begin{aligned}
\overline{h}^{S}_{\mu}(T,\cU|Y):=\lim_{\rho\to 0}\overline{h}^{S}_{\mu}(T,\cU,\rho|Y):=&\lim_{\rho\to 0}\limsup_{n\to \infty}\frac{1}{n}\int\log N_{\mu_x}(\cU^{n-1}_0,\rho)d\mu(x),\\
\underline{h}^{S}_{\mu}(T,\cU|Y):=\lim_{\rho\to 0}\underline{h}^{S}_{\mu}(T,\cU,\rho|Y):=&\lim_{\rho\to 0}\liminf_{n\to \infty}\frac{1}{n}\int\log N_{\mu_x}(\cU^{n-1}_0,\rho)d\mu(x).
\end{aligned}
\end{equation*}

\begin{proposition}\label{shapirathm}
Let $\pi: (X, T)\to (Y,S)$ be a factor map between two TDSs, $\mu\in \cE_T(X)$ and $\cU\in \cC_X^o$. Then
$$\overline{h}^{S}_{\mu}(T,\cU|Y)=\underline{h}^{S}_{\mu}(T,\cU |Y)=h_\mu(T,\cU|Y).$$
\end{proposition}
The proof of the proposition adapts the ideas from \cite{Sha} to the relative case which is provided in the appendix.
In the following we write
$$h^{S}_{\mu}(T,\cU|Y):=\overline{h}^{S}_{\mu}(T,\cU|Y)=\underline{h}^{S}_{\mu}(T,\cU |Y).$$

\begin{question}
Let $\mu\in \cE_T(X)$ and $\cU\in \cC_X^o$. Do we have for every $0<\rho<1$,
$$\overline{h}^{S}_{\mu}(T,\cU,\rho|Y)=\underline{h}^{S}_{\mu}(T,\cU,\rho|Y)?$$
\end{question}
\begin{remark}
The answer is positive in the classical case (i.e., when $Y$ is trivial) by Theorem 4.2 in \cite{Sha}, but there are some difficulties to extend the proof (see p.234 in \cite{Sha}) to the relative case. For example, it seems not easy to obtain an estimate analogous to \eqref{e:firstcount} below due to the multiple Rohlin towers.
\end{remark}

\begin{theorem}\label{VPShapira}
Let $T: (X, d)\to (X,d)$ and $S:Y\to Y$ be two TDSs and $\pi:X\to Y$ a factor map. Then
 for any potential $\vp \in C(X, \RR)$,
\begin{equation*}
\begin{aligned}
&\overline{\mdim}_M(X,d,T,\vp|Y)\\
=&\limsup_{\e \to 0}\frac{1}{\log\frac{1}{\e}}\max_{\mu\in \cE_T(X)}\inf_{\diam \cU\le \e}\left(h^S_\mu(T,\cU|Y)+\log\frac{1}{\e}\int \vp d\mu\right)
\end{aligned}
\end{equation*}
and
\begin{equation*}
\begin{aligned}
&\underline{\mdim}_M(X,d,T,\vp|Y)\\
=&\liminf_{\e \to 0}\frac{1}{\log\frac{1}{\e}}\max_{\mu\in \cE_T(X)}\inf_{\diam \cU\le \e}\left(h^S_\mu(T,\cU|Y)+\log\frac{1}{\e}\int \vp d\mu\right).
\end{aligned}
\end{equation*}
\end{theorem}
\begin{proof}
At first for any $\e>0$ and $\mu\in \cE_T(X)$, by Proposition \ref{shapirathm}
\begin{equation*}
\begin{aligned}
&\inf_{\diam \cU\le \e}(h^S_\mu(T,\cU|Y)+\log\frac{1}{\e}\int \vp d\mu)\\
= &\inf_{\diam \cU\le \e}(h_\mu(T,\cU|Y)+\log\frac{1}{\e}\int \vp d\mu)\\
=&\inf_{\diam \cU\le \e}\inf_{\a\succeq \cU}(h_\mu(T,\a|Y)+\log\frac{1}{\e}\int \vp d\mu)\\
\ge &\inf_{\diam \a\le \e}(h_\mu(T,\a|Y)+\log\frac{1}{\e}\int \vp d\mu).
\end{aligned}
\end{equation*}

For the other direction, we proceed as in the proof of Theorem \ref{VPpartition}. Given $\e>0$, choose $\cU=\cU_\e\in \cC^o_X$ such that $\diam \cU \le \e$ and $Leb \cU\ge \e/4$ (see Lemma 3.4 in \cite{GS}). By Theorem \ref{LVP}, (2) of Lemma \ref{coincide} and \eqref{e:difference}
\begin{equation*}
\begin{aligned}
&\sup_{\mu\in \cE_T(X)}\inf_{\diam \cU\le \e}(h^S_\mu(T,\cU|Y)+\log\frac{1}{\e}\int \vp d\mu)\\
\le &\sup_{\mu\in \cE_T(X)}(h^S_\mu(T,\cU|Y)+\log\frac{1}{\e}\int \vp d\mu)\\
= &\max_{\mu\in \cE_T(X)}(h_\mu(T,\cU|Y)+\log\frac{1}{\e}\int \vp d\mu)\\
=&P(T,\log\frac{1}{\e}\vp,\cU|Y)\\
\le &P(d, T, \varphi\log\frac{8}{\e}, \epsilon/8|Y)+\|\vp\|\log 8+\tau_\cU\log\frac{1}{\e}.
\end{aligned}
\end{equation*}
Dividing both sides by $\log\frac{1}{\e}$ and taking limit $\e\to 0$, we obtain the equalities with ``$\sup$'', which is in fact a ``$\max$'' by the first half of the proof.
\end{proof}

\subsection{Variational principle via Katok's entropy}
For $\nu\in \cM(X)$, let $N_\nu(n,\e,\rho)$ denote the minimal number of $d_n$-balls with radius $\e$ which cover a subset of $X$ of $\nu$-measure no less than $1-\rho$.
Let $\mu\in \cE_T(X)$ and then define
\begin{equation*}
\begin{aligned}
\overline{h}^{K}_{\mu}(T,\e|Y):=\lim_{\rho\to 0}\overline{h}^{K}_{\mu}(T,\e,\rho|Y):=&\lim_{\rho\to 0} \limsup_{n\to \infty}\frac{1}{n}\int \log N_{\mu_x}(n,\e,\rho)d\mu(x),\\
\underline{h}^{K}_{\mu}(T,\e|Y):=\lim_{\rho\to 0}\underline{h}^{K}_{\mu}(T,\e,\rho|Y):=&\lim_{\rho\to 0} \liminf_{n\to \infty}\frac{1}{n}\int \log N_{\mu_x}(n,\e,\rho)d\mu(x).
\end{aligned}
\end{equation*}
\begin{theorem}\label{VPKatok}
Let $T: (X, d)\to (X,d)$ and $S:Y\to Y$ be two topological dynamical systems and $\pi:X\to Y$ a factor map. Then
 for any potential $\vp \in C(X, \RR)$,
\begin{equation*}
\begin{aligned}
&\overline{\mdim}_M(X,d,T,\vp|Y)\\
=&\limsup_{\e \to 0}\frac{1}{\log\frac{1}{\e}}\max_{\mu\in \cE_T(X)}\left(\overline{h}^{K}_{\mu}(T,\e|Y)+\log\frac{1}{\e}\int \vp d\mu\right)\\
=&\limsup_{\e \to 0}\frac{1}{\log\frac{1}{\e}}\max_{\mu\in \cE_T(X)}\left(\underline{h}^{K}_{\mu}(T,\e|Y)+\log\frac{1}{\e}\int \vp d\mu\right)
\end{aligned}
\end{equation*}
and
\begin{equation*}
\begin{aligned}
&\underline{\mdim}_M(X,d,T,\vp|Y)\\
=&\liminf_{\e \to 0}\frac{1}{\log\frac{1}{\e}}\max_{\mu\in \cE_T(X)}\left(\overline{h}^{K}_{\mu}(T,\e|Y)+\log\frac{1}{\e}\int \vp d\mu\right)\\
=&\liminf_{\e \to 0}\frac{1}{\log\frac{1}{\e}}\max_{\mu\in \cE_T(X)}\left(\underline{h}^{K}_{\mu}(T,\e|Y)+\log\frac{1}{\e}\int \vp d\mu\right).
\end{aligned}
\end{equation*}
\end{theorem}
\begin{proof}
Let $0<\rho<1$ and $\mu\in \cE_T(X)$. Let $\cU$ be any finite open cover of $X$ with $\diam \cU \le \e$.
As any element of $\cU_0^{n-1}$ is contained in an $(n,\e)$-Bowen ball, we have $N_{\mu_x}(n,\e,\rho)\le N_{\mu_x}(\cU^{n-1}_0,\rho)$ for $\mu \ae x$.
Hence $\overline{h}^{K}_{\mu}(T,\e|Y)\le h^{S}_{\mu}(T,\cU|Y).$
By Theorem \ref{VPShapira},
\begin{equation*}
\begin{aligned}
&\overline{\mdim}_M(X,d,T,\vp|Y)\\
=&\limsup_{\e \to 0}\frac{1}{\log\frac{1}{\e}}\max_{\mu\in \cE_T(X)}\inf_{\diam \cU\le \e}\left(h^{S}_{\mu}(T,\cU|Y)+\log\frac{1}{\e}\int \vp d\mu\right)\\
\ge&\limsup_{\e \to 0}\frac{1}{\log\frac{1}{\e}}\sup_{\mu\in \cE_T(X)}\left(\overline{h}^{K}_{\mu}(T,\e|Y)+\log\frac{1}{\e}\int \vp d\mu\right).
\end{aligned}
\end{equation*}

More specifically, as before we can choose $\cU\in \cC_X^o$ with $\diam \cU \le \e$ and $Leb \cU\ge \e/4$. As every $(n,\e/4)$-Bowen ball is contained in some element of $\cU_0^{n-1}$, we have $N_{\mu_x}(\cU^{n-1}_0,\rho)\le N_{\mu_x}(n,\e/4,\rho)$.
Therefore
$h^{S}_{\mu}(T,\cU|Y)\le \underline{h}^{K}_{\mu}(T,\e/4|Y).$
By Theorem \ref{VPShapira},
\begin{equation*}
\begin{aligned}
&\overline{\mdim}_M(X,d,T,\vp|Y)\\
 =&\limsup_{\e \to 0}\frac{1}{\log\frac{1}{\e}}\max_{\mu\in \cE_T(X)}\inf_{\diam \cU\le \e}\left(h^{S}_{\mu}(T,\cU|Y)+\log\frac{1}{\e}\int \vp d\mu\right)\\
\le&\limsup_{\e \to 0}\frac{1}{\log\frac{1}{\e}}\sup_{\mu\in \cE_T(X)}\left(\underline{h}^{K}_{\mu}(T,\e/4|Y)+\log\frac{1}{\e}\int \vp d\mu\right)\\
=&\limsup_{\e \to 0}\frac{1}{\log\frac{1}{\e}}\sup_{\mu\in \cE_T(X)}\left(\underline{h}^{K}_{\mu}(T,\e|Y)+\log\frac{1}{\e}\int \vp d\mu\right).
\end{aligned}
\end{equation*}
where in the last equality we use the trick $\lim_{\e\to 0}\frac{\log \frac{1}{\e}}{\log \frac{4}{\e}}=1$. Thus we obtain the equality with ``$\sup$'', and then the latter half of the proof shows that it is in fact a ``$\max$''. The equalities for $\underline{\mdim}_M(X,d,T,\vp|Y)$ in the theorem can be proved analogously.
\end{proof}

\section{Variational principle via Brin-Katok local entropy}
In this section, we prove the variational principle involving upper Brin-Katok local entropy, and a variational inequality and principle involving lower Brin-Katok local entropy under certain conditions. Let us first introduce some terminology.

Let $T: X\to X$ be a TDS. Following the idea of Brin-Katok local entropy \cite{BK}, we define for $\mu\in \cM(X)$
\begin{equation*}
\underline{h}^{BK}_\mu(T,\e):=\int_X\underline{h}^{BK}_\mu(T,x,\e) d\mu, \quad \overline{h}^{BK}_\mu(T,\e):=\int_X\overline{h}^{BK}_\mu(T,x,\e) d\mu
\end{equation*}
where
\begin{equation*}\label{e:BK}
\begin{aligned}
\underline{h}^{BK}_\mu(T,x,\e):=&\liminf_{n\to \infty}-\frac{1}{n}\log\mu(B_{n}(x,\epsilon)),\\
\overline{h}^{BK}_\mu(T,x,\e):=&\limsup_{n\to \infty}-\frac{1}{n}\log\mu(B_{n}(x,\epsilon)).
\end{aligned}
\end{equation*}
Furthermore, define
\begin{equation*}
\begin{aligned}
\underline{h}^{BK}_\mu(T,x):=\lim_{\e\to 0}\underline{h}^{BK}_\mu(T,x,\e), &\quad \overline{h}^{BK}_\mu(T,x):=\lim_{\e\to 0}\overline{h}^{BK}_\mu(T,x,\e), \\
\underline{h}^{BK}_\mu(T):=\int_X\underline{h}^{BK}_\mu(T,x) d\mu, &\quad \overline{h}^{BK}_\mu(T):=\int_X\overline{h}^{BK}_\mu(T,x) d\mu.
\end{aligned}
\end{equation*}
When $\mu\in \cM_T(X)$, \emph{Brin-Katok's formula} \cite{BK} states that
$\underline{h}^{BK}_{\mu}(T,x)=\overline{h}^{BK}_{\mu}(T,x)$ for $\mu$-a.e. $x\in X,$ and
$$\underline{h}^{BK}_{\mu}(T)=\overline{h}^{BK}_{\mu}(T)=h_{\mu}(T)$$
where $h_{\mu}(T)$ is the metric entropy of $\mu$.
In particular, when $\mu\in \cE_T(X),$ $\underline{h}^{BK}_{\mu}(T,x)=\overline{h}^{BK}_{\mu}(T,x)=h_{\mu}(T)$ for $\mu$-a.e. $x\in X.$
The novelty here is that the Brin-Katok local entropy can be defined for any $\mu\in \cM(X)$.

Let $T: (X, d)\to (X,d)$ and $S:Y\to Y$ be two TDSs and $\pi:X\to Y$ a factor map. Define
\begin{equation*}
\begin{aligned}
\overline{h}^{BK}_{\mu}(T,\e|Y):=&\int \overline{h}^{BK}_{\mu_x}(T,\e)d\mu(x),\\
\underline{h}^{BK}_{\mu}(T,\e|Y):=&\int \underline{h}^{BK}_{\mu_x}(T,\e)d\mu(x).
\end{aligned}
\end{equation*}

Let $\mu\in \cE_T(X)$. Then $\overline{h}_{\mu_x}(T,x,\e)\ge \overline{h}_{\mu_{Tx}}(T,Tx,\e)$ and $\underline{h}_{\mu_x}(T,x,\e)\ge \underline{h}_{\mu_{Tx}}(T,Tx,\e)$. By Birkhoff ergodic theorem and ergodicity of $\mu$,
\begin{equation}\label{e:BKergodic}
\begin{aligned}
\overline{h}_{\mu_x}(T,x,\e)=\overline{h}^{BK}_\mu(T,\e|Y),\quad
\underline{h}_{\mu_x}(T,x,\e)=\underline{h}^{BK}_\mu(T,\e|Y)
\end{aligned}
\end{equation}
for $\mu \ae x$.
\begin{theorem}\label{VPBK}
Let $\pi:(X, d, T)\to (Y,S)$ be a factor map between two TDSs. Then
 for any potential $\vp \in C(X, \RR)$,
\begin{equation*}
\begin{aligned}
&\overline{\mdim}_M(X,d,T,\vp|Y)\\
=&\limsup_{\e \to 0}\frac{1}{\log\frac{1}{\e}}\max_{\mu\in \M_T(X)}\left(\overline{h}_\mu^{BK}(T, \e|Y)+\log\frac{1}{\e}\int \vp d\mu\right)
\end{aligned}
\end{equation*}
and
\begin{equation*}
\begin{aligned}
&\underline{\mdim}_M(X,d,T,\vp|Y)\\
=&\liminf_{\e \to 0}\frac{1}{\log\frac{1}{\e}}\max_{\mu\in \M_T(X)}\left(\overline{h}_\mu^{BK}(T, \e|Y)+\log\frac{1}{\e}\int \vp d\mu\right).
\end{aligned}
\end{equation*}
Moreover, both $\M_T(X)$ in the above equalities can be replaced by $\cE_T(X)$.
\end{theorem}

\begin{proof}
Given $\e>0$, let $\a$ be any finite Borel partition of $X$ with $\diam \a\le \e$. Then $\a_0^{n-1}(y)\subset B_n(y,\e)$ for any $y\in X$. For $\mu\ae x$ and $\mu_x \ae y$, since $\mu_x=\mu_y$ we have
\begin{equation*}
\begin{aligned}
\overline{h}^{BK}_{\mu_x}(T,y,\e)=&\limsup_{n\to \infty}-\frac{1}{n}\log\mu_x(B_{n}(y,\epsilon))
\le \limsup_{n\to \infty}-\frac{1}{n}\log\mu_x(\a_0^{n-1}(y))\\
= &\limsup_{n\to \infty}-\frac{1}{n}\log\mu_y(\a_0^{n-1}(y))
=h_{\mu}(T,\a|Y,y)
\end{aligned}
\end{equation*}
which implies that
\begin{equation*}
\begin{aligned}
\overline{h}_\mu^{BK}(T, \e|Y)=&\int\int\overline{h}^{BK}_{\mu_x}(T,y,\e)d\mu_x(y)d\mu(x)\\
\le& \int h_{\mu}(T,\a|Y,y)d\mu(y)=h_{\mu}(T,\a|Y).
\end{aligned}
\end{equation*}
It follows from Theorem \ref{VPpartition} that
\begin{equation*}
\begin{aligned}
&\overline{\mdim}_M(X,d,T,\vp|Y)\\
=&\limsup_{\e \to 0}\frac{1}{\log\frac{1}{\e}}\max_{\mu\in \M_T(X)}\inf_{\diam \a\le \e}\left(h_{\mu}(T,\a|Y)+\log\frac{1}{\e}\int \vp d\mu\right)\\
\ge &\limsup_{\e \to 0}\frac{1}{\log\frac{1}{\e}}\sup_{\mu\in \M_T(X)}\left(\overline{h}_{\mu}^{BK}(T,\e|Y)+\log\frac{1}{\e}\int \vp d\mu\right).
\end{aligned}
\end{equation*}

Let us prove the other direction. Given $\e>0$, take any $\a\in \cP_X$ with $\diam \a\le\epsilon.$
Let $\mu\in \cE_T(X)$. By \eqref{e:BKergodic}, there exist $X_1\in \mathcal{B}_X$ with
$\mu(X_1)=1$ and
$W_x=\pi^{-1}\pi(x)\cap X_1$ for any $x\in X_1$, such that
for each $x\in X_1$ and $y\in W_x,$
 $\mu_x(W_x)=1$ and
\begin{equation*}\label{eq10}
\limsup_{n\rightarrow\infty}-\frac{1}{n}\log\mu_x(B_n(y,\e))=\overline{h}^{BK}_\mu(T,\e|Y).
\end{equation*}
Fix $x\in X_1$. For $n\in \mathbb N$ and $\delta>0, $ set
$$I_n=\{y\in W_x:\mu_x(B_n(y,\e))>\exp(-(\overline{h}^{BK}_\mu(T,\e|Y)+\delta)n)\}.$$
Then $\lim_{n\to\infty}\mu_x(I_n)=1.$ Thus for any $0<\rho<1$, there exists sufficiently large $n\in \mathbb N$ such that $\mu_x(I_n)>1-\rho.$
Let $S_n$ be a maximal $(n,2\e)$ separated subset of $I_n$, then it is also an $(n,2\e)$ spanning subset of $I_n$. It follows that
$$\mu_x(\cup_{y\in I_n}B_n(y,2\e))\ge \mu_x(I_n)>1-\rho$$
and hence $\#S_n\ge N_{\mu_x}(n,2\e,\rho)$. On the other hand, since $\{B_n(y,\e):y\in S_n\}$ are disjoint and $S_n\subset I_n$, one has $$1\ge \mu_x(\cup_{y\in I_n}B_n(y,\e))=\sum_{y\in I_n}\mu(B_n(y,\e))\ge \#S_n\exp(-(\overline{h}^{BK}_\mu(T,\e|Y)+\delta)n).$$
Then
$N_{\mu_x}(n,2\e,\rho)\leq \exp((\overline{h}^{BK}_\mu(T,\e|Y)+\delta)n). $
Thus for any $\delta>0$,
$$\limsup_{n\to \infty}\frac{1}{n}\int \log N_{\mu_x^\xi}(n,\e,\rho)d\mu(x)\leq \overline{h}^{BK}_\mu(T,\e|Y)+\delta.$$
Letting $\delta\to 0$,
and then $\rho\to 0$ we have
$\overline{h}^{K}_{\mu}(T,\e|Y)\leq \overline{h}^{BK}_\mu(T,\e|Y).$
Thus by Theorem \ref{VPKatok}
\begin{equation*}
\begin{aligned}
&\overline{\mdim}_M(X,d,T,\vp|Y)\\
 =&\limsup_{\e \to 0}\frac{1}{\log\frac{1}{\e}}\max_{\mu\in \cE_T(X)}\left(\overline{h}^{K}_{\mu}(T,\e|Y)+\log\frac{1}{\e}\int \vp d\mu\right)\\
\le &\limsup_{\e \to 0}\frac{1}{\log\frac{1}{\e}}\sup_{\mu\in \cE_T(X)}\left(\overline{h}^{BK}_\mu(T,\e|Y)+\log\frac{1}{\e}\int \vp d\mu\right).
\end{aligned}
\end{equation*}
Thus the equality holds with ``$\sup$'', which in fact is a ``$\max$'' by the very last inequality above. The part for $\underline{\mdim}_M(X,d,T,\vp|Y)$ can be proved similarly.
\end{proof}

The following proposition illustrates a nontrivial relation between Brin-Katok local entropy and entropy of partitions given in SMB Theorem \ref{smb0} for conditional measures.
\begin{proposition}\label{second}
Suppose that $T:X\to X$ is a TDS, $\mu\in \cM_T(X)$ and $\xi$ is a $T$-invariant measurable partition of $X$. For any $\rho>0$ and any $\a\in \cP_X$ with $M:=\#\a>1$, there exist $\d>0$ and a measurable subset $I\subset X$ satisfying $\mu(I)>1-\rho^{\frac{1}{4}}$ such that: For any $x\in I$, there exists a measurable subset $D_x\subset X$ such that $\mu_{x}^{\xi}(D_{x})>1-5\rho^{\frac{1}{4}}$ and
\begin{equation*}\label{e:second}
\begin{aligned}
\underline{h}^{BK}_{\mu_{x}^{\xi}}(T,y,\d)\geq h_{\mu}(T, \alpha|\xi)(y)-3(\Delta+\rho)
\end{aligned}
\end{equation*}
for any $y\in D_{x}$, where
$$\Delta=\Delta(\rho):=2\sqrt{\rho}\log(M-1)-2\sqrt{\rho}\log 2\sqrt{\rho}-(1-2\sqrt{\rho})\log (1-2\sqrt{\rho}).$$
\end{proposition}
\begin{remark}\label{BKchoice}
The proof of the proposition is given in \cite{Zhou} (see p.4069 there) when $\mu$ is ergodic, which adapts Brin-Katok's original ideas \cite{BK} to conditional entropy. In the proof, the only requirement for the choice of $\d$ is that
$\mu(U_\d(\partial\a))<\rho$,
where $U_\d(\partial\a)$ is the $\d$-neighborhood of the boundary of $\a$, i.e.,
$U_\d(\partial\a):=\{x\in X: B(x,\d)\not\subseteq\a(x)\}$.

\end{remark}
\begin{theorem}\label{VPBK1}
Let $T: (X, d)\to (X,d)$ and $S:Y\to Y$ be two topological dynamical systems and $\pi:X\to Y$ a factor map. Then
 for any potential $\vp \in C(X, \RR)$,
\begin{equation*}
\begin{aligned}
&\overline{\mdim}_M(X,d,T,\vp|Y)\\
\le &\limsup_{\e \to 0}\frac{1}{\log\frac{1}{\e}}\sup_{\mu\in \cE_T(X)}\left((\overline{\dim}_BX+1)\underline{h}_\mu^{BK}(T, \e|Y)+\log\frac{1}{\e}\int \vp d\mu\right)
\end{aligned}
\end{equation*}
and
\begin{equation*}
\begin{aligned}
&\underline{\mdim}_M(X,d,T,\vp|Y)\\
\le &\liminf_{\e \to 0}\frac{1}{\log\frac{1}{\e}}\sup_{\mu\in \cE_T(X)}\left((\underline{\dim}_BX+1)\underline{h}_\mu^{BK}(T, \e|Y)+\log\frac{1}{\e}\int \vp d\mu\right).
\end{aligned}
\end{equation*}
\end{theorem}
\begin{proof}
Given $\e>0$, choose $\mu_\e\in \cE_T(X)$ such that the maximum value
$$\max_{\mu\in \cE_T(X)}\inf_{\diam \a\le \e}\left(h_{\mu}(T,\a|Y)+\log\frac{1}{\e}\int \vp d\mu\right)$$
in Theorem \ref{VPpartition} is achieved at $\mu_\e$.

Given any $\rho_1>0$ small enough (independent of $\e$), pick $\rho>0$ small enough such that $(1-\rho^\frac{1}{4})(1-5\rho^\frac{1}{4})>1-\rho_1$ and $3(\Delta(\rho)+\rho)<\rho_1$. Note that $\rho$ is independent of $\e$.
Let $\{B(x_i, \e/2): i=1, \cdots, N\}$ be an open cover of $X$ where $N=N(X,\e/2)$ is the least number of balls of radius $\e/2$ needed to cover $X$. Denote $k:=\lceil\frac{1}{\rho}\rceil$. For each $i$, the $\frac{\e}{4kN}$-open neighborhoods of $\partial B(x_i, \e/2+j\cdot\frac{\e}{2kN}), j=0,\cdots, kN-1$ are disjoint, and therefore there exists a $j_i$ such that the $\frac{\e}{4kN}$-open neighborhoods of $\partial B(x_i, \e/2+j_i\cdot\frac{\e}{2kN})$ has $\mu_\e$-measure less than $\frac{1}{kN}\le \rho/N$. Then we get an open cover $\{B(x_i, \e/2+j_i\cdot\frac{\e}{2kN}): i=1, \cdots, N\}$ of $X$, which induces naturally a finite Borel partition $\a_\e$ with $\mu_\e(U_{\frac{\e}{4kN}}(\partial\a_\e))<\rho$. 

By Remark \ref{BKchoice}, we can apply Proposition \ref{second} to get
\begin{equation*}
\begin{aligned}
&\underline{h}_{\mu_\e}^{BK}(T, \frac{\e}{4kN}|Y)= \underline{h}^{BK}_{(\mu_\e)_x}(T,y,\frac{\e}{4kN})\\
\ge&  h_{\mu_\e}(T,\a_\e|Y,y)-\rho_1=h_{\mu_\e}(T,\a_\e|Y)-\rho_1,
\end{aligned}
\end{equation*}
where we used $\eqref{e:BKergodic}$ and Theorem \ref{smb0} as $\mu_\e$ is ergodic.
It follows that
\begin{equation*}
\begin{aligned}
&\overline{\mdim}_M(X,d,T,\vp|Y)\\
=&\limsup_{\e \to 0}\frac{1}{\log\frac{1}{\e}}\max_{\mu\in \cE_T(X)}\inf_{\diam \a\le \e}\left(h_{\mu}(T,\a|Y)+\log\frac{1}{\e}\int \vp d\mu\right)\\
\le &\limsup_{\e \to 0}\frac{1}{\log\frac{1}{\e}}\left(h_{\mu_\e}(T,\a_\e|Y)+\log\frac{1}{\e}\int \vp d\mu_\e\right)\\
\le &\limsup_{\e \to 0} \left(\frac{1}{\log\frac{1}{\e}}\left(\underline{h}_{\mu_\e}^{BK}(T,\frac{\e}{4kN}|Y)+\rho_1\right)+\int \vp d\mu_\e\right)\\
\le &\limsup_{\e \to 0} \sup_{\mu\in \cE_T(X)}\left(\frac{\log\frac{4kN}{\e}}{\log\frac{1}{\e}}\frac{1}{\log\frac{4kN}{\e}}\underline{h}_{\mu}^{BK}(T,\frac{\e}{4kN}|Y)+\int \vp d\mu\right)\\
= &\limsup_{\e \to 0}\sup_{\mu\in \cE_T(X)} \left((\overline{\dim}_BX+1)\frac{1}{\log\frac{1}{\e}}\underline{h}_{\mu}^{BK}(T,\e|Y)+\int \vp d\mu\right)\\
= &\limsup_{\e \to 0}\frac{1}{\log\frac{1}{\e}}\sup_{\mu\in \cE_T(X)} \left((\overline{\dim}_BX+1)\underline{h}_{\mu}^{BK}(T,\e|Y)+\log\frac{1}{\e}\int \vp d\mu\right)
\end{aligned}
\end{equation*}
where  we used the fact that
$\limsup_{\e\to 0}\frac{\log\frac{4kN}{\e}}{\log\frac{1}{\e}}=\overline{\dim}_BX+1.$ 

This proves the inequality for $\overline{\mdim}_M(X,d,T,\vp|Y)$. The equality for\\ $\underline{\mdim}_M(X,d,T,\vp|Y)$ can be proved similarly.
\end{proof}

For $0<\e,\rho<1$ and $\mu\in \cM(X)$, denote
$$\d_{\e,\rho, \mu}:=\sup\{\d>0: \exists\a\in \cP_X \text{\ such that\ } \diam \a<\e, \mu(U_{\d}(\partial\a))<\rho\}$$
and $\d_{\e,\rho}:=\inf_{\mu\in \cM(X)}\d_{\e,\rho, \mu}$.
The following definition is a variant of the ones introduced in \cite[Definition 4.1]{KKVW} and \cite{TWL}.
\begin{definition}
A compact metric space $(X,d)$ is called \emph{very well-partitionable} if $\lim_{\e\to 0}\frac{\log \frac{1}{\d_{\e,\rho}}}{\log \frac{1}{\e}}=1$ for any $0<\rho<1$.
\end{definition}

As a corollary of the proof of Theorem \ref{VPBK1}, we have
\begin{theorem}\label{VPBK2}
Let $\pi:(X, d, T)\to (Y,S)$ be a factor map between two TDSs and $(X,d)$ is very well-partitionable. Then
 for any potential $\vp \in C(X, \RR)$,
\begin{equation*}
\begin{aligned}
&\overline{\mdim}_M(X,d,T,\vp|Y)\\
=&\limsup_{\e \to 0}\frac{1}{\log\frac{1}{\e}}\max_{\mu\in \M_T(X)}\left(\underline{h}_\mu^{BK}(T, \e|Y)+\log\frac{1}{\e}\int \vp d\mu\right)
\end{aligned}
\end{equation*}
and
\begin{equation*}
\begin{aligned}
&\underline{\mdim}_M(X,d,T,\vp|Y)\\
=&\liminf_{\e \to 0}\frac{1}{\log\frac{1}{\e}}\max_{\mu\in \M_T(X)}\left(\underline{h}_\mu^{BK}(T, \e|Y)+\log\frac{1}{\e}\int \vp d\mu\right).
\end{aligned}
\end{equation*}
Moreover, both $\M_T(X)$ in the above equalities can be replaced by $\cE_T(X)$.
\end{theorem}
We discuss the box dimension of the following important example. Let $X=[0,1]^\ZZ$ be the infinite product of the unit interval $[0,1]$ and $T: X \to X$ the shift: $T((x_n)_{n\in \ZZ})=(x_{n+1})_{n\in \ZZ}$. We define a distance $d$ on $X$ by
$$d(x, y) =\sum_{n\in \ZZ}2^{-|n|}|x_n-y_n|$$
for $x =(x_n)_{n\in \ZZ}$ and $y=(y_n)_{n\in \ZZ}$ in $X$. It is calculated in \cite{LT1} that
$$N(X,d, \e)=(1+\lfloor12/\e\rfloor)^{2\lceil\log 4/\e\rceil+2}$$
which implies $\underline{\dim}_B(X,d)=\infty$. So the upper bound in Theorem \ref{VPBK1} is far from being optimal for this example. On the other hand, by the calculation in \cite[Example 5.3]{Shi}, in fact we have
$$\lim_{\e \to 0}\frac{1}{\log\frac{1}{\e}}\underline{h}_\mu^{BK}(T, \e)=1=\mdim_M(X,d,T)$$
where $\mu=\cL^{\otimes \ZZ}$ where $\cL$ is the Lebesgue measure on $[0,1]$.

\begin{question}\cite[Open problem]{Shi}
In the setting of Proposition \ref{VPBK1}, do we have
\begin{equation*}
\begin{aligned}
&\overline{\mdim}_M(X,d,T,\vp|Y)\\
= &\limsup_{\e \to 0}\frac{1}{\log\frac{1}{\e}}\max_{\mu\in \M_T(X)}\left(\underline{h}_\mu^{BK}(T, \e|Y)+\log\frac{1}{\e}\int \vp d\mu\right)
\end{aligned}
\end{equation*}
and similarly for $\underline{\mdim}_M(X,d,T,\vp|Y)$?
\end{question}
\section{Relative metric mean dimension with potential after Ledrappier-Walters}
Let $T: (X, d)\to (X,d)$ and $S:Y\to Y$ be two TDSs and $\pi:X\to Y$ a factor map. In this section, we always consider $\nu\in \cM_S(Y)$ and a potential $\vp \in C(X, \RR)$.

\begin{definition}\label{unstableentropy2}
The \emph{upper relative metric mean dimension with potential} $\vp$ of the system $(X,d,T,Y, S, \nu)$ is defined by
\begin{equation*}
\begin{aligned}
\overline{\mdim}^\nu_M(X,d,T,\vp|Y):=\limsup_{\e \to 0}\frac{1}{\log\frac{1}{\e}}\int_YP(d, T, \log\frac{1}{\e}\varphi, \epsilon,y)d\nu(y)
\end{aligned}
\end{equation*}
where
\begin{equation*}
\begin{aligned}
P(d, T, \log\frac{1}{\e}\varphi, \epsilon,y):=\limsup_{n\to \infty}\frac{1}{n}\log P_n(d, T, \log\frac{1}{\e}\varphi, \e, \pi^{-1}y).
\end{aligned}
\end{equation*}
Similarly, the \emph{lower relative metric mean dimension with potential} $\vp$ of the system $(X,d,T,Y,\nu)$ is defined by
\begin{equation*}
\begin{aligned}
\underline{\mdim}^\nu_M(X,d,T,\vp|Y):=\liminf_{\e \to 0}\frac{1}{\log\frac{1}{\e}}\int_YP(d, T, \log\frac{1}{\e}\varphi, \epsilon,y)d\nu(y).
\end{aligned}
\end{equation*}
\end{definition}


A local inner variational principle has been obtained in \cite{MaChen}.
\begin{theorem}[Cf. Theorem 1.2 in \cite{MaChen}]
We have
\begin{equation*}
\begin{aligned}
\int_Y P(d,T,\vp,\cU,y)d\nu(y)=\max_{\mu\in \M_T(X)}\left\{h_\mu(T,\cU|Y)+\int \vp d\mu: \pi\mu=\nu\right\}.
\end{aligned}
\end{equation*}
\end{theorem}

We then have the following inner variational principle involving entropy of partitions.
\begin{theorem}\label{VPpartition1}
Let $T: (X, d)\to (X,d)$ and $S:Y\to Y$ be two TDSs and $\pi:X\to Y$ a factor map. Then
 for any $\nu\in \cM_S(Y)$ and any potential $\vp \in C(X, \RR)$,
\begin{equation*}
\begin{aligned}
&\overline{\mdim}^\nu_M(X,d,T,\vp|Y)\\
=&\limsup_{\e \to 0}\frac{1}{\log\frac{1}{\e}}\max_{\mu\in \M_T(X),\pi\mu=\nu}\inf_{\diam \a\le \e}\left(h_\mu(T,\a|Y)+\log\frac{1}{\e}\int \vp d\mu\right)
\end{aligned}
\end{equation*}
and
\begin{equation*}
\begin{aligned}
&\underline{\mdim}^\nu_M(X,d,T,\vp|Y)\\
=&\liminf_{\e \to 0}\frac{1}{\log\frac{1}{\e}}\max_{\mu\in \M_T(X), \pi\mu=\nu}\inf_{\diam \a\le \e}\left(h_\mu(T,\a|Y)+\log\frac{1}{\e}\int \vp d\mu\right).
\end{aligned}
\end{equation*}
\end{theorem}

\begin{proof}
We only need modify the proof of Theorem \ref{VPpartition}.
Given $\e>0$, choose $\cU=\cU_\e\in \cC^o_X$ such that $\diam \cU \le \e$ and $Leb \cU\ge \e/4$. Then
\begin{equation*}
\begin{aligned}
&\sup_{\mu\in \M_T(X),\pi\mu=\nu}\inf_{\diam \a\le \e}(h_\mu(T,\a|Y)+\log\frac{1}{\e}\int \vp d\mu)\\
\le &\max_{\mu\in \M_T(X),\pi\mu=\nu}\inf_{\a\succeq \cU}(h_\mu(T,\a|Y)+\log\frac{1}{\e}\int \vp d\mu)\\
=&\int_YP(d, T,\vp\log\frac{1}{\e},\cU,y)d\nu(y)\\
=&\int_Y\lim_{n\to \infty}\frac{1}{n}\log p_n(d, T, \varphi\log\frac{1}{\e}, \cU, \pi^{-1}y)d\nu(y)\\
\le & \int_Y\lim_{n\to \infty}\frac{1}{n}\log (1/\e)^{n\tau_\cU}P_n(d, T, \varphi\log\frac{1}{\e}, Leb\cU/2, \pi^{-1}y)d\nu(y)\\
\le & \int_Y\lim_{n\to \infty}\frac{1}{n}\log (1/\e)^{n\tau_\cU}P_n(d, T, \varphi\log\frac{8}{\e}, \e/8, \pi^{-1}y)d\nu(y)+\|\vp\|\log 8\\
=&\int_YP(d, T, \log\frac{8}{\e}\varphi, \epsilon,y)d\nu(y)+\tau_\cU\log\frac{1}{\e}+\|\vp\|\log 8.
\end{aligned}
\end{equation*}
On the other hand,
\begin{equation*}
\begin{aligned}
&\int_YP(d, T, \vp\log\frac{1}{\e}, \epsilon,y)d\nu(y)\\
= & \int_Y\lim_{n\to \infty}\frac{1}{n}\log P_n(d, T, \vp\log\frac{1}{\e}, \e, \pi^{-1}y)d\nu(y)\\
\le&\int_Y\lim_{n\to \infty}\frac{1}{n}\log p_n(d, T, \vp\log\frac{1}{\e}, \cU, \pi^{-1}y)d\nu(y)\\
=&\int_Y P(T,\log\frac{1}{\e}\vp,\cU,y)d\nu(y)\\
= &\max_{\mu\in \M_T(X), \pi\mu=\nu}\inf_{\a\succeq \cU}(h_\mu(T,\a|Y)+\log\frac{1}{\e}\int \vp d\mu)\\
\le &\sup_{\mu\in \M_T(X), \pi\mu=\nu}\inf_{\diam \a\le \e/8}(h_\mu(T,\a|Y)+\log\frac{1}{\e}\int \vp d\mu).
\end{aligned}
\end{equation*}
Dividing both sides by $\log\frac{1}{\e}$ and taking the limit as $\e\to 0$, we obtain the equality in the theorem with $\sup$, which becomes a ``$\max$'' by the last inequality above. The part for $\underline{\mdim}^\nu_M(X,d,T,\vp|Y)$ can be proved similarly.
\end{proof}

\begin{remark}
From the above proof of Theorem \ref{VPpartition1}, we have
$$P(d, T, \vp\log\frac{1}{\e}, \epsilon,y)\le P(d, T, \varphi\log\frac{1}{\e}, \cU_\epsilon,y)\le P(d, T, \varphi\log\frac{1}{\e}, \epsilon/8,y)$$
and hence
\begin{equation*}
\begin{aligned}
\overline{\mdim}^\nu_M(X,d,T,\vp|Y)=\limsup_{\e \to 0}\frac{1}{\log\frac{1}{\e}}\int_YP(d, T, \varphi\log\frac{1}{\e}, \cU_\epsilon,y)d\nu(y).
\end{aligned}
\end{equation*}
By \cite[Lemma 4.7]{MaChen}, for any $\cU\in \cC_X^o$,
$$P(d,T,\vp,\cU|Y)=\max_{\nu\in \M_S(Y)}\int P(d,T,\vp,\cU,y)d\nu(y).$$
So we can deduce the variational principle Theorem \ref{VPpartition} by taking $\max_{\nu\in \M_S(Y)}$ in each expression in the above proof of Theorem \ref{VPpartition1}.
\end{remark}

Let $\nu\in \cE_T(X)$. All the equalities in Theorems \ref{VPShapira}, \ref{VPKatok}, \ref{VPBK} still hold if
$$\overline{\mdim}_M(X,d,T,\vp|Y), \overline{\mdim}_M(X,d,T,\vp|Y), \max_{\mu\in \cE_T(X)}$$ are replaced by $$\overline{\mdim}^\nu_M(X,d,T,\vp|Y), \overline{\mdim}^\nu_M(X,d,T,\vp|Y), \max_{\mu\in \cE_T(X),\pi\mu=\nu}$$ respectively.
Indeed, in the proof of these three theorems only the relations between entropy involving partitions, Shapira's entropy, Katok's entropy and Brin-Katok local entropy are involved. Thus we can modify the proof accordingly to obtain the corresponding inner variational principles.
\\

$\mathbf{Acknowledgements}$. The author would like to thank Daren Wei and Ruxi Shi for helpful discussions. This work is supported by NSFC Nos. 11701559 and 12071474.

\section{Appendix: Proof of Proposition \ref{shapirathm}}
We prepare two lemmas before going to the proof of Proposition \ref{shapirathm}.
\begin{lemma}\label{partition}
Let $\pi: (X, T)\to (Y,S)$ be a factor map, $\mu\in \cM_T(X)$ and $\mu=\int \mu_yd\nu(y)$ the disintegration of $\mu$ over $\nu=\pi \mu$. Then for any $\cV\in \cC_X$ and $0<\rho<1$, there exists $\b\in \cP_X$ such that $\b\succeq \cV$ and $N_{\mu_y}(\b, \rho)\le N_{\mu_y}(\cV, \rho)$ for $\nu \ae y\in Y$.
\end{lemma}
\begin{proof}
Let $\mathcal{V}=\{V_1, \cdots, V_m\}$. For $\nu \ae y\in Y$, there exists $I_y \subset \{1,\cdots, m\}$ with cardinality $N_{\mu_y}(\cV, \rho)$ such that $\bigcup_{i\in I_y}V_i$ covers a subset of $\pi^{-1}y$ up to a set of $\mu_y$-measure less than $\rho$. Hence we can find $y_1, \cdots, y_s \in Y$ such that for $\nu \ae y\in Y$, $I_y =I_{y_i}$ for some $i\in \{1,\cdots,s\}$. For $i = 1,\cdots s$, define
$$D_i=\{y\in Y: \mu_y(\bigcup_{j\in I_{y_i}}V_j)>1-\rho\}.$$
Let $C_1=D_1$, $C_i = D_i\setminus \cup_{j=1}^{i-1}D_j, i= 2,\cdots, s$. 

Fix $i\in \{1,\cdots, s\}$. Assume $I_{y_i}=\{k_1 < \cdots < k_{t_i}\}$ where $t_i=N_{\mu_{y_i}}(\cV, \rho)$. Take $\{W_1(y_i),\cdots,W_{t_i}(y_i)\}$ where
$$W_1(y_i) =V_{k_1}, W_2(y_i) =V_{k_2}\setminus V_{k_1},\cdots, W_{t_i}(y_i)=V_{k_{t_i}}\setminus \cup_{j=1}^{t_i-1}V_{k_j}.$$
Define $A:=X\setminus \left(\cup_{i=1}^s(\pi^{-1}C_i\cap \cup_{j=1}^{t_i}W_j(y_j)\right)$ and $A_1=A\cap V_1, A_l:=A\cap (V_l\setminus \cup_{j=1}^{l-1}V_j), l=2, \cdots, m$. Finally, define
\begin{equation*}
\begin{aligned}\b=\{&\pi^{-1}C_1\cap W_1(y_1), \cdots, \pi^{-1}C_1\cap W_{t_1}(y_1), \cdots, \\
&\pi^{-1}C_s\cap W_1(y_s), \cdots, \pi^{-1}C_s\cap W_{t_s}(y_s), A_1, \cdots, A_m\}.
\end{aligned}
\end{equation*}
Then $\b\succeq \cV$ and $N_{\mu_y}(\b, \rho)\le N_{\mu_y}(\cV, \rho)$ for $\nu \ae y\in Y$.
\end{proof}

\begin{lemma}[Strong Rohlin Lemma, see Lemma 2.5 in \cite{Sha}]\label{Rohlin}
Let $(X, \cB, \mu, T)$ be an ergodic, aperiodic invertible system and let $\a\in \cP_X$. Then for any $\d>0$ and $n\in \NN$, one can find a set $B\in \BB$ such that $B, TB, \cdots, T^{n-1}B$ are mutually disjoint, $\mu(\cup_{i=0}^{n-1}T^iB)>1-\d$ and the distribution of $\a$ is the same as the distribution of the partition $\a|_B$ that $\a$ induces on $B$.
\end{lemma}

\begin{proof}[Proof of Proposition \ref{shapirathm}]
The proposition follows from the following two lemmas.
\end{proof}

\begin{lemma}
Let $\mu\in \cE_T(X)$ and $\cU\in \cC_X^o$. Then for any $0<\rho<1$,
$$\overline{h}^{S}_{\mu}(T,\cU,\rho|Y)\leq h_\mu(T,\cU|Y).$$
\end{lemma}
\begin{proof}
Take any finite Borel partition $\a\succeq \cU$.
According to Theorem \ref{smb0}, as $\mu$ is ergodic,
there exist $Y_1\in \mathcal{B}_Y$ with
$\nu(Y_1)=1$ such that
for each $y\in Y_1$ and $\mu_y\ae x,$
\begin{equation*}\label{eq10}
\lim_{n\rightarrow\infty}\frac{-\log\mu_y(\alpha_0^{n-1}(x))}{n}= h_\mu(T,\a|Y).
\end{equation*}
Fix $y\in Y_1$. For $n\in \mathbb N$ and $\delta>0, $ set
$$I_n:=\{x\in \pi^{-1}y:\mu_y(\alpha_0^{n-1}(x))>\exp(-(h_\mu(T,\a|Y)+\delta)n)\}
=\pi^{-1}y\cap \bigcup_{V\in \mathcal{J}_n}V,
$$
where $\mathcal{J}_n=\{V\in \alpha_0^{n-1}:\mu_y(V)>\exp(-(h_\mu(T,\a|Y)+\delta)n)$.
Then for any $\delta>0$, $\lim_{n\to\infty}\mu_y(I_n)=1.$ Thus, for sufficiently large $n\in \mathbb N$, we have $\mu_y(I_n)>1-\rho.$
Since
\begin{align*}
\begin{split}
\#\mathcal{J}_n&=\#\{V\in \alpha_0^{n-1}:\mu_y(V)>\exp(-(h_\mu(T,\a|Y)+\delta)n)\\
&\leq \exp((h_\mu(T,\a|Y)+\delta)n),
\end{split}
\end{align*}
the set $I_n$ can be covered by at most $\exp((h_\mu(T,\a|Y)+\delta)n)$ elements of the partition $\alpha_0^{n-1}$. Then
$$N_{\mu_y}(\cU_0^{n-1},\rho)\le N_{\mu_y}(\a_0^{n-1},\rho)\leq \exp((h_\mu(T,\a|Y)+\delta)n). $$
Thus for any $\delta>0$,
$$\limsup_{n\to \infty}\frac{1}{n}\int \log N_{\mu_y}(\cU_0^{n-1},\rho)d\mu(y)\leq h_\mu(T,\a|Y)+\delta.$$
Letting $\delta\to 0$, we obtain
$\overline{h}^{S}_{\mu}(T,\cU,\rho|Y)\leq h_\mu(T,\a|Y).$
Taking infimum over $\a\succeq \cU$, we have
$$\overline{h}^{S}_{\mu}(T,\cU,\rho|Y)\leq h_\mu(T,\cU|Y).$$
\end{proof}


\begin{lemma}\label{lessthan}
Let $\mu\in \cE_T(X)$ and $\cU\in \cC_X^o$. Then $$h_\mu(T,\cU|Y)\le \underline{h}^{S}_{\mu}(T,\cU|Y).$$
\end{lemma}
\begin{proof}
Let $\mu\in \cE_T(X)$. If the system $(X,T)$ is periodic, then $\mu$ is supported on a fixed point of $T$ and
$$\underline{h}^{S}_{\mu}(T,\cU|Y)=h_\mu(T,\cU|Y)=0.$$
Thus let us assume $(X,T)$ is aperiodic.

Fix $n\in \NN$. Let $\b$ be constructed as in the proof of Lemma \ref{partition} for $\cV\succeq \cU_0^{n-1}$. We also use the notation from that proof, for example, $A$ is the subset of $X$ such that $\mu(A)<\rho$ and for any $x\notin A$, $N_{\mu_x}(\b, \rho)\le N_{\mu_x}(\cU_0^{n-1}, \rho)$. Choose $\d>0$ such that $0<\rho+\d<1/4$. By Lemma \ref{Rohlin}, we can construct a strong Rohlin tower with respect to $\b$, with height $n$ and error $<\d$. Let $\tilde B$ denote the base of the tower and $B=\tilde B \setminus A$. Clearly, $\mu (B)>(1-\rho)\mu (\tilde B)$ and $\mu(E)\ge 1-(\rho+\d)$ where $E=\cup_{i=0}^{n-1}T^iB$. Consider $\b|_{\tilde B}$ and index its elements by sequences $i_0, \cdots, i_{n-1}$ such that if $B_{i_0, \cdots, i_{n-1}}\in \b|_{\tilde B}$, then $T^jB_{i_0, \cdots, i_{n-1}} \subset U_{i_j}$ for every $0\le j\le n-1$. Let $\hat \a:=\{\hat A_1, \cdots, \hat A_M\}$ be a partition of $E$ defined by
$$\hat A_m:=\cup\{T^jB_{i_0, \cdots, i_{n-1}}: 0\le j\le n-1, i_j=m\}.$$
Note that $\hat A_m \subset U_m$ for every $1\le m\le M$. Extend $\hat \a$ to a partition $\a$ of $X$ in some way such that $\a \succeq \cU$ and $\#\a=2M$.

Set $\eta^4=\rho+\d$ and define for every $k>n$ large enough, $f_k(x)=\frac{1}{k}\sum_{i=0}^{k-1}\chi_E(T^ix)$ and $L_k:=\{x\in X: f_k(x)>1-\eta^2\}$.
Then by Birkhoff ergodic theorem $\int f_k >1-\eta^4$, and
 $$\eta^2\cdot \mu(L_k^c)\le \int_{L_k^c}1-f_k\le \int_{X}1-f_k\le \eta^4$$
which gives $\mu(L_k)\ge 1-\eta^2$.
Put $J_k$ to be the set of $x\in X$ such that for any $j\ge k$,
\begin{equation}\label{e:birkhoff2}
\mu_x(\alpha_0^{j-1}(x))<\exp(-(h_\mu(T,\a|Y)-\eta)j),
\end{equation}
and
\begin{equation}\label{e:cover}
|\frac{1}{j}\sum_{i=0}^{j-1}\log N_{\mu_{T^{i}x}} (\cU_0^{n-1},\rho) \chi_B(T^ix)-\int_B \log N_{\mu_{z}} (\cU_0^{n-1},\rho) d\mu(z)|\le \eta.
\end{equation}
By Theorem \ref{smb0} and the Birkhoff ergodic theorem, $\mu(J_k)>1-\eta^2$ for $k$ large enough. Set $G_k=L_k\cap J_k$ and then $\mu(G_k)>1-2\eta^2$.
Define $\tilde G_k=\{x\in G_k:\mu_x(G_k)\geq 1-4\eta\},$
then
$$\tilde G_k^c=\{x\in G_k:\mu_x(G_k)<1-4\eta\}\cup G_k^c=\{x\in G_k:\mu_x(G^c_k)>4\eta\}\cup G_k^c.$$
Therefore,
$$ \mu(\tilde G_k^c)\cdot 4\eta\leq\int\mu_x(G_k^c)d\mu(x)+\mu(G_k^c)=2\mu(G_k^c)\le 4\eta^2,$$
i.e., $\mu(\tilde G_k^c)\le \eta$.

We fix an element $C_y$ of this partition of $G_k\cap \pi^{-1}\pi y$ and want to estimate the number of $\a_0^{n-1}$-elements needed to cover it. If $0 \le i_1 < \cdots < i_m\le k-n$ are the
times elements of $C_y$ visit $B$, then we need at most $N_{\mu_{T^{i_j}y}} (\cU_0^{n-1},\rho)$ $\a_{i_j}^{i_j+n-1}$-elements
to cover $C_y$. Because the size of $[0, k-1]\setminus \cup_j[i_j, i_j+n-1]$ is at most $\eta^2 k+2n$,
we need at most $\prod_{j=1}^mN_{\mu_{T^{i_j}y}} (\cU_0^{n-1},\rho)\cdot (2M)^{\eta^2 k+2n}$ $\a_0^{k-1}$-elements to cover $C_y$. Finally, in view of $\eqref{e:cover}$, we know that $G_k\cap \pi^{-1}\pi y$ can be covered by no more than
\begin{equation}\label{e:firstcount}
e^{kH(\eta^2+2n/k)}\cdot (2M)^{\eta^2 k+2n}\cdot e^{k(\int_B \log N_{\mu_{z}} (\cU_0^{n-1},\rho)d\mu(z)+\eta)}
\end{equation}
$\a_0^{k-1}$-elements.
Since $y\in G_k\subset J_k$, any $V\in \a_0^{k-1}$ intersecting nontrivially with $G_k\cap \pi^{-1}\pi y$ has $\mu_y$-measure less than
$\exp(-(h_\mu(T,\a|Y)-\eta)k)$ by \eqref{e:birkhoff2}. Thus we have
\begin{equation}\label{e:count1}
\begin{aligned}
&1-4\eta\le \mu_y(G_k\cap \pi^{-1}\pi y)\\
\le &e^{-(h_\mu(T,\a|Y)-\eta)k}e^{kH(\eta^2+2n/k)}\cdot (2M)^{\eta^2 k+2n}
\cdot   e^{k(\int_B \log N_{\mu_{z}} (\cU_0^{n-1},\rho)d\mu(z)+\eta)}.
\end{aligned}
\end{equation}
Recall that the distribution of $\b$ is the same as the distribution of the partition $\b|_{\tilde B}$ and $z\mapsto N_{\mu_z} (\cU_0^{n-1},\rho)$ is constant on each atom of $\b|_{X\setminus A}$ by Lemma \ref{partition}.
Then by \eqref{e:count1} and setting $k\to \infty$, we get
\begin{equation*}
\begin{aligned}
&h_\mu(T,\a|Y)\\
\le &\eta+H(\eta^2)+\eta^2 \log (2M)+\int_B \log N_{\mu_{z}} (\cU_0^{n-1},\rho)d\mu(z)+\eta\\
\le &2\eta+H(\eta^2)+\eta^2 \log (2M)+\frac{1}{n}\int \log N_{\mu_{z}} (\cU_0^{n-1},\rho) d\mu(z).
\end{aligned}
\end{equation*}
By letting $\d\to 0$, we obtain
\begin{equation*}
\begin{aligned}
h_\mu(T,\a|Y)
\le 2\rho^{\frac{1}{4}}+H(\rho^{\frac{1}{2}})+\rho^{\frac{1}{2}} \log (2M)+\frac{1}{n}\int \log N_{\mu_{z}} (\cU_0^{n-1},\rho) d\mu(z).
\end{aligned}
\end{equation*}
Taking $\liminf_{n\to \infty}$ and then $\lim_{\rho\to 0}$, we have
\begin{equation*}
\begin{aligned}
h_\mu(T,\cU|Y)\le h_\mu(T,\a|Y)\le \underline{h}^{S}_{\mu}(T,\cU|Y).
\end{aligned}
\end{equation*}
\end{proof}

\end{document}